\documentclass[12pt]{article}
\input epsf.tex

\usepackage{amsmath}
\usepackage{amsthm}
\usepackage{amsfonts}
\usepackage{amssymb}
\usepackage{graphicx}
\usepackage{latexsym}

\usepackage{amsmath,amsthm,amsfonts,amssymb}

\usepackage{epsfig}

\usepackage{amssymb}
\usepackage[all]{xy}
\xyoption{poly}
\usepackage{fancyhdr}
\usepackage{wrapfig}

\theoremstyle{plain}
\newtheorem{thm}{Theorem}[section]
\newtheorem{prop}[thm]{Proposition}
\newtheorem{lem}[thm]{Lemma}

\theoremstyle{definition}
\newtheorem{defn}{Definition}
\theoremstyle{remark}
\newtheorem{remark}{Remark}

\advance \topmargin by -\headheight
\advance \topmargin by -\headsep
\evensidemargin \oddsidemargin
\def\cc{{\curvearrowright}}

\def\F{{\mathbb F}}

\def\bG{{\bar G}}

\def\N{{\mathbb N}}

\def\chix{{\raise.5ex\hbox{$\chi$}}}

\def\bX{{\overline X}}

\def\bY{{\overline Y}}

\def\Cay{{\textrm{Cay}}}

\newcommand{\ilim}{\mathop{\varprojlim}\limits}
\newcommand{\End}{\operatorname{End}}

\begin{document}
\title{Integrable orbit equivalence rigidity for free groups}
\author{Lewis Bowen\footnote{supported in part by NSF grant DMS-1500389, NSF CAREER Award DMS-0954606} \\ University of Texas at Austin}
\maketitle
\begin{abstract}
It is shown that every accessible group which is integrable orbit equivalent to a free group is virtually free. Moreover, we also show that any integrable orbit-equivalence between finitely generated groups extends to their end compactifications.
\end{abstract}

\noindent

\noindent
\tableofcontents

\section{Introduction}

Measure equivalence (ME) is an equivalence relation on groups, introduced by M. Gromov  \cite{gromov-1993} as a measure-theoretic counterpart to quasi-isometry. Most of the research in this area has focussed on rigidity phenomena. For example, Furman proved \cite{furman-1999a, furman-1999b} that any group ME to a lattice in a higher rank simple Lie group has a finite normal subgroup whose quotient is commensurable to a lattice in the same Lie group. See \cite{furman-me-survey} for a survey of further results.

Here we consider the class of free groups. This class is far from rigid: there is a large variety of groups measure-equivalent to a free group \cite{gaboriau-2005} and we do not even have a conjectural classification of such groups. So it makes sense to consider the more restrictive notion of measure-equivalence known as {\bf integrable orbit equivalence} (IOE) or $L^1$-orbit-equivalence ($L^1$-OE) which takes into account the metric in addition to measure-theoretic structure. This notion is stricter than integrable measure equivalence (IME) (also called $L^1$-measurable equivalence) that first appeared in \cite{bader-furman-sauer-2013b} where it was shown that any group that is $L^1$-ME to a lattice in $SO(n,1)$ ($n\ge 3$) has a finite normal subgroup whose quotient is a lattice in $SO(n,1)$. Another milestone is T. Austin's work proving $L^1$-ME rigidity for nilpotent groups \cite{austin-bowen-2013} (see also M. Cantrell's recent strengthening of Austin's results \cite{cantrell-2015}). 

The main result of this paper is that any finitely generated accessible group that admits a strict integrable embedding into a free group is virtually free. In particular, any finitely generated accessible group that is IOE to a free group is virtually free. These terms are defined next.

\subsection{Accessible groups}

According to Stallings' Ends Theorem \cite{stallings-1968a}, if a finitely generated group $\Gamma$ has more than one end then it splits as either a nontrivial free product with amalgamation  or as an HNN extension over a finite subgroup. If such splittings cannot occur indefinitely, then the group is called {\em accessible}. C.T.C. Wall conjectured \cite{wall-1971} that all finitely generated groups are accessible. A counterexample was obtained by Dunwoody \cite{dunwoody-1993}. However, Dunwoody showed that all finitely presented groups are accessible \cite{dunwoody-1985}. 

\subsection{Strict integrable embeddings}

Let $\Gamma,G$ be finitely generated groups. Intuitively, a strict integrable embedding of $\Gamma$ into $G$ is a random map from $\Gamma$ into $G$ that is `Lipschitz on average' and has a bounded number of preimages. To be precise, fix a finite symmetric generating set $S_G$ of $G$ and define the {\bf word length} of any $g\in G$ by $|g|_G:=n$ where $n$ is the smallest natural number such that there exist $s_1,\ldots, s_n \in S_G$ with $g=s_1\cdots s_n$. 

Given an action of $\Gamma$ on a set $X$, a {\bf cocycle} into $G$ is a map $\alpha:\Gamma \times X \to G$ such that
$$\alpha(g_1g_2,x) = \alpha(g_1,g_2x)\alpha(g_2,x)\quad \forall g_1,g_2 \in \Gamma, x\in X.$$
In the case of concern, $X$ is endowed with a probability measure $\mu$, the action $\Gamma \cc (X,\mu)$ is measure-preserving and $\alpha$ is measurable. Then we say that $\alpha$ is {\bf integrable} if
$$\int |\alpha(g,x)|_G~d\mu(x) < \infty$$
for every $g\in \Gamma$. While the precise value of $\int |\alpha(g,x)|_G~d\mu(x)$ depends on the generating set $S_G$, its finiteness does not and therefore integrability of $\alpha$ does not depend on $S_G$. 

We say that $\alpha$ is a {\bf strict integrable embedding} if in addition to being integrable there is a constant $C>0$ such that for every $h\in G$,
$$\#\{g\in \Gamma:~\alpha(g,x)=h\}\le C$$
for a.e. $x$. This notion is more restrictive than the notion of integrable embedding defined in the appendix of \cite{austin-bowen-2013}. 

Our main result is:
\begin{thm}\label{thm:main}
Let $\Gamma$ be a finitely generated accessible group. If $\Gamma$ admits a strict integrable embedding into a free group then $\Gamma$ is virtually free.
\end{thm}

\begin{defn}
Two groups $\Gamma,\Lambda$ are {\bf integrably orbit equivalent} (IOE) if there exist probability measure-preserving essentially free ergodic actions $\Gamma \cc (X,\mu)$ and $\Lambda \cc (X,\mu)$ such that for a.e. $x\in X$, $\Gamma x = \Lambda x$ and the orbit cycles, defined by
$$\alpha:\Gamma \times X \to \Lambda, \quad \alpha(g,x)x = gx,$$
$$\beta:\Lambda \times X \to \Gamma, \quad \beta(h,x)x = hx$$
are integrable. 
\end{defn}
Clearly an IOE cocycle is a strict integrable embedding. So Theorem \ref{thm:main} implies that any finitely generated accessible group IOE to a free group is virtually free.

We do not know whether accessibility is a necessary condition nor whether `strict integrable embedding' can be weakened to `integrable embedding' or IOE weakened to IME.

{\bf Acknowledgements}. This work was inspired from conversations with Tim Austin, Uri Bader, Alex Furman and Roman Sauer.

\section{Preliminaries}

\begin{defn}
If $X$ is a connected locally connected $\sigma$-compact topological space then let
$$\End(X):=\ilim_K \pi_0(X\setminus K).$$
To be precise, the inverse limit is over all compact subsets $K \subset X$ and $\pi_0(X\setminus K)$ denotes the set of noncompact connected components of $X\setminus K$. We give $\pi_0(X\setminus K)$ the discrete topology. If $K \subset L $ are compact subsets of $X$ then there is a natural map from $\pi_0(X \setminus L)$ to $\pi_0(X\setminus K)$ and $\End(X)$ is the inverse limit of this system (where the collection of compacts of $X$ is ordered by inclusion). In particular, for every compact $K$, there is a natural map $\pi_K:\End(X) \to \pi_0(X \setminus K)$.

The {\bf end compactification of $X$}, denoted $\bX$ is the space $\bX :=X \cup \End(X)$ with the following topology: every open subset of $X$ is $\bX$. Also, for every compact $K \subset X$ and $C \in \pi_0(X\setminus K)$, the set 
$$C \cup \{\xi \in \End(X):~\pi_K(\xi) \in C\} \subset \bX$$
is open. These sets form a basis for the topology of $\bX$.
\end{defn}

Let $\Gamma$ be a group with a finite generating set $S$. The {\bf Cayley graph of $(\Gamma,S)$}, denoted $\Cay(\Gamma,S)$, has vertex set $\Gamma$ and edge set $\{(g,gs):~g\in G,s\in S\}$.  We usually let $\bar{\Gamma}$ denote the end  compactification $\overline{\Cay(\Gamma,S)}$, leaving the generating set implicit.


It is well-known that any finitely generated group quasi-isometric to a free group is itself virtually free. This is usually attributed to Gromov via Stallings' Ends Theorem. Alternatively, it follows from Thomassen-Woess \cite{thomassen-woess-1993} that accessibility is a quasi-isometric-invariant and from Papasoglu-Whyte \cite{papasoglu-whyte-2002} that any accessible group quasi-isometric to a free group must be virtually free. Indeed, this implies more:
\begin{thm}\label{thm:v-free}
If $\Gamma$ is the fundamental group of a finite graph of groups in which all vertex and edge groups are finite then $\Gamma$ is virtually free.
\end{thm}
\begin{proof}
This follows from \cite{papasoglu-whyte-2002} although it may have been known earlier.
\end{proof}

We also note:
\begin{lem}\label{lem:qi}
If $\Gamma$ is the fundamental group of a finite graph of groups in which all edge groups are finite and $\Gamma_v\le \Gamma$ is a vertex subgroup then $\Gamma_v$ is quasi-isometrically embedded in $\Gamma$.
\end{lem}

Finally, we introduce a notion of $L^1$-embedding:
\begin{defn}
Let $\Gamma \cc (X,\mu)$ a probability measure-preserving action and $\alpha:\Gamma \times X \to G$ a measurable cocycle. We say $\alpha$ is an {\em $L^1$-embedding} if
\begin{itemize}
\item $\alpha$ is $L^1$: for every $g\in \Gamma$, $\int |\alpha(g,x)|_G~d\mu(x)<\infty$ where $|\cdot |_G$ denotes length with respect to a fixed word metric on $G$;
\item there is a constant $C>0$ such that for any $h\in G$,
$$\# \{g\in \Gamma:~ \alpha(g,x) = h\} \le C$$
for a.e. $x$.
\end{itemize}
\end{defn}

\begin{remark}
It is straightforward to check that a composition of $L^1$-embeddings is an $L^1$-embedding and that any cocycle arising from an $L^1$-OE is an $L^1$-embedding.
\end{remark}

In order to prove Theorem \ref{thm:main}, it now suffices to show the following. Let $\Gamma \cc (X,\mu)$ be a probability measure-preserving action and $\alpha:\Gamma \times X \to G$ an $L^1$-embedding. Then $\Gamma$ is not 1-ended. We will show this in the next section.


\section{The Space of Ends}

\begin{thm}\label{thm:ends}
Let $\Gamma, G$ be finitely generated groups, $\Gamma \cc (X,\mu)$ a probability measure-preserving action, $\alpha:\Gamma \times X \to G$ an $L^1$-embedding. Define $\alpha':\Gamma \times X \to G$ by
$$\alpha'(g,x)=\alpha(g^{-1},x)^{-1}.$$
Let $\bar{\Gamma}, \bG$ denote the end-compactifications of $\Gamma, G$ respectively with respect to fixed finite generating subsets. Then $\alpha'$ extends to a map, also denoted by $\alpha'$ from $\bar{\Gamma}  \times X \to \bar{G}$ such that
\begin{itemize}
\item $\alpha'_{gx}(g\xi) = \alpha(g,x)\alpha'_x(\xi)$ (for $g\in G, \xi \in \bar{\Gamma}$ and a.e. $x\in X$);
\item $\alpha'_x:\bar{\Gamma} \to \bar{G}$ is continuous for a.e. $x$.
\end{itemize}
Here, $\alpha'_x(g):= \alpha'(g,x)$.
\end{thm}

\begin{remark}
The Theorem above implies that every finitely generated group is 1-taut relative to its space of ends, in the terminology of \cite{bader-furman-sauer-2013b}. We will not need this fact.
\end{remark}

Theorem \ref{thm:ends} follows immediately from the next two lemmas.

\begin{lem} \label{l:ends}
Let $X,Y$ be connected locally connected $\sigma$-compact topological spaces.
Let $\alpha: X\to Y$ be a continuous map.
Assume that for every compact $K\subset Y$ there exists a compact $F\subset X$ such that
$\alpha(X\setminus F)\subset Y\setminus K$ and $\alpha$ descends to a well defined map $\pi_0(X\setminus F)\to \pi_0(Y\setminus K)$.
Then $\alpha$ extends continuously to $\bar{\alpha}: \bX \to \bY$.
\end{lem}

\begin{proof}
For every $K\subset Y$ we have a map
\[ \End(X)\to \pi_0(X\setminus D)\to \pi_0(Y\setminus K), \]
so the lemma follows by the definition of the inverse limit $\ilim_D \pi_0(Y\setminus K)$.
\end{proof}

\begin{lem} \label{lem:prox}
For a.e. $x\in X$ and every finite set $K\subset G$
there exists a finite set $F\subset \Gamma$ (depending on $x$ and $K$)
such that $\alpha'_x(\Gamma \setminus F) \subset G\setminus K$ and $\alpha'_x$ descends to a map $\pi_0(\Gamma\setminus F)\to \pi_0(G\setminus K)$. 
\end{lem}

\begin{proof}
Let $S_\Gamma, S_G$ be finite generating sets for $\Gamma,G$ respectively. Let $|\cdot|_\Gamma, |\cdot|_G$ denote word length on $\Gamma,G$ respectively.

For each $h_1,h_2 \in G$, choose a geodesic segment $\gamma[h_1,h_2]$ from $h_1$ to $h_2$. More precisely, for every integer $0\le n \le |h_1^{-1}h_2|_G$, there is an element $\gamma[h_1,h_2](n) \in G$ so that 
$$\gamma[h_1,h_2](n)^{-1}\gamma[h_1,h_2](n+1) \in S_G$$
if $n<|h_1^{-2}h_2|_G$ and $\gamma[h_1,h_2](0)=h_1, \gamma[h_1,h_2]( |h^{-1}h_2|_G) = h_2$. Let us also require that this choice is left-invariant so that $h \gamma[h_1,h_2]=\gamma[hh_1,hh_2]$ for any $h,h_1,h_2 \in G$.

For each $x\in X$, $g\in \Gamma$, $s\in S_\Gamma$, we imagine an airplane flying from $\alpha'_x(g)$ to $\alpha'_x(gs)$. The path of the flight is the geodesic $\gamma[\alpha'_x(g),\alpha'_x(gs)]$. We call this an {\bf $s$-flight}. For $k\in G$, we let $F_{s,k}(x)$ denote the set of elements $g\in \Gamma$ such that the $s$-flight from $\alpha'_x(g)$ to $\alpha'_x(gs)$ contains $k$. That is: 
$$F_{s,k}(x) := \{g\in \Gamma:~ k \in \gamma[\alpha'_x(g), \alpha'_x(gs)]\}.$$

\noindent {\bf Claim 1}. $F_{s,k}(x)$ is finite for a.e. $x$. In fact,
$$\int \#F_{s,k}(x)~d\mu(x) \le C\int |\alpha(s^{-1},x)^{-1}|_G~d\mu(x) < \infty$$
where $C>0$ is the constant in the definition of $L^1$-embedding.

\begin{proof}[Proof of Claim 1]
It suffices to show that $\int \# F_{s,k}(x)~d\mu(x) < \infty$. In order to prove this, let
$$L_{s,k} = \{(g,x) \in \Gamma \times X:~ g^{-1} \in F_{s,k}(x)\}.$$
Let $c_\Gamma$ denote the counting measure on $\Gamma$. Then $\int \#F_{s,k}~d\mu = c_\Gamma \times \mu(L_{s,k})$. Because the action $\Gamma \cc (X,\mu)$ is invariant, 
$$c_\Gamma \times \mu(L_{s,k}) = c_\Gamma \times \mu(R_{s,k})$$
where $R_{s,k} = \{ (g^{-1},gx):~ (g,x) \in L_{s,k}\}$. By definition
\begin{eqnarray*}
c_\Gamma \times \mu(R_{s,k}) &=& \int \# \{g\in \Gamma:~ (g,x)\in R_{s,k} \}~d\mu(x).
\end{eqnarray*}
However, $(g,x) \in R_{s,k}$ if and only if $(g^{-1},gx)\in L_{s,k}$ if and only if $g\in F_{s,k}(gx)$ if and only if $k \in \gamma[\alpha'_{gx}(g), \alpha'_{gx}(gs)]$ if and only if
$$\alpha'_{gx}(g)^{-1}k \in \gamma[e, \alpha'_{gx}(g)^{-1}\alpha'_{gx}(gs)].$$
Let us now compute
\begin{eqnarray*}
\alpha'_{gx}(g)^{-1}\alpha'_{gx}(gs) = \alpha(g^{-1}, gx)\alpha(s^{-1}g^{-1}, gx)^{-1} = \alpha(s^{-1},x)^{-1}
\end{eqnarray*}
by the cocycle equation. So
\begin{eqnarray*}
\int \# F_{s,k}~d\mu &=& c_\Gamma \times \mu(R_{s,k}) = \int \#\{g\in \Gamma:~\alpha'_{gx}(g)^{-1}k \in \gamma[e, \alpha(s^{-1},x)^{-1}]\}~d\mu(x)\\
&\le& C\int |\alpha(s^{-1},x)^{-1}|_G~d\mu(x) < \infty.
\end{eqnarray*}
\end{proof}

Now let $K \subset G$ be finite and define
$$F_K(x): = \bigcup \{F_{s,k}:~s\in S_\Gamma, k\in K\}.$$
To finish the proof of the lemma, it suffices to show that if $g_1,g_2 \in \Gamma$ are in the same connected component of $\Gamma \setminus F_{K}(x)$ then $\alpha'_x(g_1), \alpha'_x(g_2)$ are in the same connected component of $G \setminus K$. Because $S_\Gamma$ is a generating set, we may assume that $g_2=g_1s$ for some $s\in S_\Gamma$. Because $g_1 \notin F_K(x)$, it follows that
$$K \cap \gamma[\alpha'_x(g_1), \alpha'_x(g_1s)] = \emptyset.$$
So $\alpha'_x(g_1), \alpha'_x(g_1s)$ are in the same connected component of $G \setminus K$ as required.
\end{proof}

\begin{defn}
Suppose $H$ is a finitely generated group and $S_H\subset H$ is a finite symmetric generating set.  Let $\Cay(H,S_H)$ be the associated Cayley graph. Given a subset $F \subset H$, let $\partial F$ be the set of all edges $e$ of $\Cay(H,S_H)$ with one endpoint in $F$ and one endpoint in $H \setminus F$. 
\end{defn}

\begin{lem}\label{lem:2-end}
Suppose $H$ is a finitely generated group and $S_H\subset H$ is a finite symmetric generating set.  Suppose there exists a constant $C>0$ and finite subsets $F_n \subset H$ such that
\begin{itemize}
\item $|\partial F_n| \le C$ for all $n \in \N$,
\item $\lim_{n\to\infty} |F_n| =\infty$.
\end{itemize}
Then $H$ has at least 2 ends.
\end{lem}

\begin{proof}

We identify each $F_n$ with its induced subgraph in $\Cay(H,S_H)$.  We may assume without loss of generality that every connected component of the complement $\Cay(H,S_H) \setminus F_n$ is infinite. This is because we may add all finite components of $\Cay(H,S_H) \setminus F_n$ to $F_n$ without increasing the size of its boundary. 

Choose elements $g_n \in F_n$, $s_n \in S_H$ so that
\begin{itemize}
\item $(g_n, g_ns_n) \in \partial F_n$
\item if $F^\circ_n \subset F_n$ is the connected component of $F_n$ containing $g_n$ then $\lim_{n\to\infty} |F_n^\circ| = +\infty$
\item there exists an infinite path $p_n \subset \Cay(H,S_H) \setminus F_n$ starting from $g_ns_n$.
\end{itemize}
Let $F'_n = g_n^{-1}F_n^\circ$ and $p'_n = g_n^{-1}p_n$. After passing to a subsequence if necessary, we may assume that $F'_n$ converges to a limit $F'_\infty$ and $p'_n$ converges to a limit $p'_\infty$ (in the topology of uniform convergence on compact subsets). We observe that $F'_\infty$ is infinite, $p'_\infty \subset \Cay(H,S_H) \setminus F'_\infty$ is an infinite path and $|\partial F'_\infty| \le C$. Thus the compact set $K:=\partial F'_\infty$ is such that there are at least two infinite components of $\Cay(H,S_H)\setminus K$ (namely, the component containing $p'_\infty$ and the component containing $F'_\infty$). This proves that $H$ has at least two ends. 

\end{proof}

\begin{prop}\label{prop:1-end}
Suppose $\Gamma$ is an infinite finitely generated group, $G=\F_r$ be a nonabelian free group, $\Gamma \cc (X,\mu)$ a probability measure-preserving action and $\alpha:\Gamma \times X \to G$ an $L^1$-embedding. Then $\Gamma$ has more than one end.
\end{prop}

\begin{proof}

We fix a free generating set of $G$ from which we obtain a word metric and a Cayley graph (which is a regular tree since $G$ is a free group). We also fix a finite generating set $S_\Gamma$ for $\Gamma$.

To obtain a contradiction, we assume that $\End(\Gamma)=\{\xi\}$ is a singleton. Define $\phi:X \to \End(G)$ by $\phi(x)=\alpha'(\xi,x)$ where $\alpha'$ is as in Theorem \ref{thm:ends}. By Theorem \ref{thm:ends},
$$\phi(hx) = \alpha(h,x) \phi(x).$$
For $n\in \N, x\in X$, let $G(n,x)$ be the set of all $g \in G$ such that $(g| \phi(x))_e \le n$ where $(\cdot | \cdot)_e$ is the Gromov product. To be precise $(g | \phi(x))_e=d(e,m)$ where, if $g\ne e$, $m \in G$ is the `midpoint' of the geodesic triangle with vertices $\{g,\phi(x),e\}$. That is, $m$ is the unique element contained in all three geodesic sides of the triangle with vertices $\{g,\phi(x),e\}$. If $g=e$ then by definition $m=e$. Thus $G(n,x)$ is the set of all elements $g \in G$ such that the geodesic from $g$ to $\phi(x)$ contains a point of distance no more than $n$ from $e$. Let $F(n,x)=\{h\in \Gamma:~ \alpha'(h,x) \in G(n,x)\}$. 
\\
\noindent {\bf Claim 1}. $\int |\partial F(n,x)| ~d\mu(x) \le C\sum_{s\in S_\Gamma} \int |\alpha(s,x)|_G~d\mu(x) =:M$. Note $M$ is independent of $n$.

\begin{proof}[Proof of Claim 1]
Let $r_n(x) \in G$ be the unique element satisfying
$$d(e,r_n(x)) = n = (r_n(x)| \phi(x))_e.$$
In other words, $n \mapsto r_n(x)$ is the geodesic from $e$ to $\phi(x)$. Observe that $\partial G(n,x)$ is the unique edge from $r_n(x)$ to $r_{n+1}(x)$.

By definition $\partial F(n,x)$ consists of all edges of the form $(g,gs)$ such that $g \in F(n,x)$ and $gs \notin F(n,x)$ ($s\in S$). Equivalently, $\alpha'_x(g) \in G(n,x)$ and $\alpha'_x(gs) \notin G(n,x)$. Equivalently, the $s$-flight from $\alpha'_x(g)$ to $\alpha'_x(gs)$ flies over $r_n(x)$. The claim now follows as in the proof of Lemma \ref{lem:prox}, Claim 1. 
\end{proof}

\noindent {\bf Claim 2}. For every $n\in \N$ and a.e. $x\in X$, $|F(n,x)| <\infty$.

\begin{proof}
The previous claim implies $\partial F(n,x)$ is finite for a.e. $x$. Because $\Gamma$ is 1-ended, for a.e. $x\in X$ either $F(n,x)$ or $\Gamma \setminus F(n,x)$ is finite. Because $\alpha': \bar{\Gamma}\times X \to \bar{G}$ is continuous and $\alpha'(\xi,x) \notin \overline{G(n,x)} = \overline{\alpha'(F(n,x),x)}$, it follows that $\Gamma \setminus F(n,x)$ must be infinite and therefore $F(n,x)$ is finite.
\end{proof}

Observe that $G(n,x)\subset G(n+1,x)$ and $\cup_{n\ge 0} G(n,x) = G$. Therefore $F(n,x) \subset F(n+1,x)$ for all $n$ and $\cup_{n\ge 0} F(n,x) = \Gamma$ which in particular implies that $\lim_{n\to\infty} |F(n,x)| = +\infty$. 

Because
$$\lim_{n\to\infty} \int |F(n,x)|~d\mu(x) = +\infty, \quad \int |\partial F(n,x)|~d\mu(x) \le M,$$
we can choose $x_n \in X$ so that $|F(n,x_n)| \to \infty$ while $|\partial F(n,x_n)|$ stays bounded. Lemma \ref{lem:2-end} now implies that $\Gamma$ has at least 2 ends, a contradiction.  
\end{proof}

\begin{proof}[Proof of Theorem \ref{thm:main}]
By assumption there exists an $L^1$-embedding $\alpha:\Gamma \times X \to G$ and $G$ is a free group. Since $\Gamma$ is accessible, we may write it as the fundamental group of a finite graph of groups in which each edge group is finite and each vertex group has $\le 1$ end. By Lemma \ref{lem:qi} each vertex group $H$ quasi-isometrically embeds into $\Gamma$. So if we restrict $\alpha$ to $H \times X$, it is still an $L^1$-embedding. So Proposition \ref{prop:1-end} implies that $H$ is not 1-ended. So every vertex group and edge group in the graph of groups decomposition of $\Gamma$ is finite. This implies that $\Gamma$ is virtually free by Theorem \ref{thm:v-free}.

\end{proof}

\bibliography{biblio}
\bibliographystyle{alpha}









\end{document}